\documentclass[12pt]{amsart}
\usepackage{amsmath,amssymb,amsfonts,amsthm,amsopn}
\usepackage{pb-diagram}
\usepackage{graphicx}
\usepackage{xcolor}

\newtheorem{theorem}{Theorem}[section]
\newtheorem{lemma}[theorem]{Lemma}
\newtheorem{corollary}[theorem]{Corollary}
\newtheorem{proposition}[theorem]{Proposition}

\newtheorem{remark}[theorem]{Remark}

\newtheorem{notation}[theorem]{Notation}

\theoremstyle{definition}
\newcommand{\beqa}{\begin{eqnarray*}}
\newcommand{\eeqa}{\end{eqnarray*}}

\newcommand{\field}[1]{\mathbb{#1}}
\newcommand{\bR}{\field{R}}

\def\cF{\mathcal{F}}

\def\rd{\bR^n}

\def\<{\left<}
\def\>{\Big)>}

\def\mv1{M_v^1}

\def\mn{(m,n)}
\def\mn'{(m',n')}

\def\Ren{\mathbb{R}^n}

\def\tauhz0{\widehat{\mathcal{T}}^\hbar(z_0)}
\def\tauhz{\widehat{\mathcal{T}}^\hbar(z)}

\def\Sn2{S_{2}(L^{2}(\Ren))}
\def\S1{S_{1}(L^{2}(\Ren))}
\def\sig00{\sigma_{0,0}}

\input{tcilatex}

\begin{document}
\title[On the symplectic covariance and interferences] {On the symplectic covariance and interferences of time-frequency
distributions}
\author{Elena Cordero}
\address{Universit\`a di Torino, Dipartimento di Matematica, via Carlo
Alberto 10, 10123 Torino, Italy}
\email{elena.cordero@unito.it}
\thanks{}
\author{Maurice de Gosson}
\address{University of Vienna, Faculty of Mathematics,
Oskar-Morgenstern-Platz 1 A-1090 Wien, Austria}
\email{maurice.de.gosson@univie.ac.at}
\thanks{}
\author{Monika D\"{o}rfler}
\address{University of Vienna, Faculty of Mathematics,
Oskar-Morgenstern-Platz 1 A-1090 Wien, Austria}
\thanks{}
\email{monika.doerfler@univie.ac.at}
\author{Fabio Nicola}
\address{Dipartimento di Scienze Matematiche, Politecnico di Torino, corso
Duca degli Abruzzi 24, 10129 Torino, Italy}
\email{fabio.nicola@polito.it}
\thanks{}
\date{}

\begin{abstract}
We study the covariance property of quadratic time-frequency distributions
with respect to the action of the extended symplectic group. We show how
covariance is related, and in fact in competition, with the possibility of
damping the interferences which arise due to the quadratic nature of the
distributions. We also show that the well known fully covariance property of
the Wigner distribution in fact characterizes it (up to a constant factor)
among the quadratic distributions $L^{2}(\mathbb{R}^{n})\rightarrow C_{0}({%
\mathbb{R}^{2n}})$. A similar characterization for the %
{closely} related Weyl transform is given as well. The
results are illustrated by several numerical experiments for the Wigner and
Born-Jordan distributions of the sum of four Gaussian functions in the
so-called \textquotedblleft diamond configuration\textquotedblright .
\end{abstract}

\maketitle

%\title[$\hbar$-Gabor frames for modulation spaces and applications] {$\hbar$-Gabor frames for modulation spaces and applications to Schr\"{o}dinger equations}

%\curraddr{\"{y} }

%\curraddr{\"{y} }

%\curraddr{\"{y} }

%\curraddr{\"{y} }

\section{Introduction}

The importance of alternatives to the Wigner transform in both
time-frequency analysis and quantum mechanics {should} not be
underestimated. Recent work in signal processing has shown that it may be
advantageous to use variants of it to reduce unwanted interference effects, 
\cite{oehl,pikula,turunen}, while seems that one particular distribution, closely
related to the physicist's Born--Jordan quantization, plays an essential
(and not yet fully understood) role in quantum mechanics. It turns out that,
luckily enough, all these transforms are particular cases of what is
commonly called the \textquotedblleft Cohen class\textquotedblright\cite%
{Cohen1,Cohenbook,book}; this class consists of all transforms $Qf$ obtained
from the Wigner distribution $Wf$ {by convolving} with a
distribution $\theta\in\mathcal{S}^{\prime}(\mathbb{R}^{2n})$: $%
Qf=Wf\ast\theta$. In the present paper we set out to analyze and discuss the
relative advantages of such general transforms from the point of view of
their symplectic covariance properties and their effect on the interferences
due to the cross-terms $Q(f,g)$ in the non-additivity of $Q$: 
\begin{equation}  \label{heiwe}
Q(f+g)=Qf+Qg+2\func{Re}Q(f,g).
\end{equation}
Precisely, let $\widehat{T}(z_{0})=e^{-\frac{i}{\hslash}\sigma(\hat{z}%
,z_{0})}$, $z_0=(x_0,p_0)\in\mathbb{R}^{2n}$, be the Heisenberg operator: 
\begin{equation*}
\widehat{T}(z_{0})f(x)=e^{\tfrac{i}{\hslash}(p_{0}x-\tfrac{1}{2}%
p_{0}x_{0})}f(x-x_{0}),\quad x\in\mathbb{R}^n.
\end{equation*}
Every distribution $Qf$ in the Cohen class enjoys the covariance property
with respect to time-frequency shifts: 
\begin{equation}  \label{uno-i}
Q(\widehat{T}(z_{0})f)(z)=Q f(z-z_0),\quad z_0, z\in\mathbb{R}^{2n}.
\end{equation}
However, in general such a distribution does not enjoy covariance with
respect to all symplectic transformations of the time-frequency plane: if $%
\widehat{S}$ is an element of the metaplectic group (regarded as a unitary
operator in $L^2(\mathbb{R}^n)$) and $S\in\mathrm{Sp}(n,\mathbb{R})$ is the
corresponding symplectic transformation we prove (Proposition \ref{pro36})
that 
\begin{equation}  \label{due-i}
Q(\widehat{S}f)(z)=Wf\ast(\theta\circ S)(S^{-1}z),\quad \forall f\in\mathcal{%
S}(\mathbb{R}^n),\quad z\in\mathbb{R}^{2n}.
\end{equation}
The Wigner distribution corresponds to the case $\theta=\delta$. Since, in
that case, $\theta\circ S=\theta$, from \eqref{uno-i}, \eqref{due-i} we
recapture the well known symplectic covariance of the Wigner distribution
with respect to the \textit{extended symplectic group}, i.e.\ the semidirect
product $\mathbb{R}^{2n}\rtimes \mathrm{Sp}(n,\mathbb{R})$.

As we will see, the covariance property and the reduction of interferences
are related in a subtle manner. For instance, as already observed, if $Qf=Wf$
we have full symplectic covariance; this implies, in particular, that one
cannot eliminate or damp the interference effect by a symplectic rotation of
the coordinates. If we choose instead the Born--Jordan distribution, which
corresponds to 
\begin{equation*}
\theta=\mathcal{F}_{\sigma}\Big(\mathrm{sinc} \left(\frac {px}{2\hbar}\right)%
\Big)
\end{equation*}
(where $\mathcal{F}_{\sigma}$ is the symplectic Fourier transform defined in %
\eqref{sympft}), we {lose} covariance with respect to the
subgroup of the symplectic group consisting of \textquotedblleft symplectic
shears\textquotedblright (see Corollary \ref{cor42} below),
but this allows us at the same time to dampen the interference effects by
rotating the coordinate system. In the general case we have similarly a
trade-off between cross-interferences and the symmetry group of $Qf$.

We illustrate graphically the compared attenuation effects of the Wigner
transform and of the Born--Jordan--Wigner transform by considering a
so-called \textquotedblleft diamond state\textquotedblright\ consisting of
four Gaussians; such structures have been studied by Zurek \cite{Zurek} in
the Wigner case in relation with the appearance of sub-quantum effects in
the theory of Gaussian superpositions. We make it clear that these effects
are greatly attenuated in the Born--Jordan case; we refer to \cite{cogoni16} for an explanation in terms of wave front sets.

Formula \eqref{due-i} motivates the study of the class of temperate
distributions $\theta\in\mathcal{S}^{\prime }(\mathbb{R}^{2n})$ such that $%
\theta\circ S=\theta$ for every $S\in \mathrm{Sp}(n,\mathbb{R})$. We
completely characterize such a class (Proposition \ref{teosimpinv}), and
thereafter use the result to characterize the Wigner distribution. Namely,
we propose a precise formulation of the folklore statement that the Wigner
distribution has special covariance properties among {all}
quadratic time-frequency distributions: we prove (Theorem \ref{teonew})
that, up to a constant factor,

\smallskip \textit{The Wigner transform is the only quadratic time-frequency
distribution $L^2(\mathbb{R}^n)\to C_0({\mathbb{R}^{2n}})$ which enjoys
covariance with respect to the extended symplectic group}.

\smallskip Here $C_0({\mathbb{R}^{2n}})$ is the space of continuous function
on the time-frequency plane which vanish at infinity; this decay condition
is essential to rule out distributions $Qf=Wf\ast \theta$ with $\theta=%
\mathrm{constant}$, which are in fact fully covariant too. In spite of the primary role of the
Wigner distribution in Time-frequency Analysis and Mathematical Physics and
the immense work of mathematicians, physicists and engineers, it seems that
the above clean characterization has never appeared in the literature.

Similarly, we provide a characterization in terms of covariance with respect
to the extended symplectic group for the Weyl transform, i.e.\ the linear
operator $A:\mathcal{S}(\mathbb{R}^n)\to\mathcal{S}^{\prime }(\mathbb{R}^n)$
defined by $\langle Af,f\rangle=\langle a, Wf\rangle$, for some symbol $a\in%
\mathcal{S}^{\prime }({\mathbb{R}^{2n}})$. Let us emphasize that the well
known characterization of the Weyl transform (\cite[Sections 7.5-7.6, pages
578-579]{stein} and \cite[Theorem 30.2]{wong}) involves instead the
non-extended symplectic group, but requires an additional condition (see
Remark \ref{rem00} below), which is instead dropped in our result. Hence our
characterization partially intersects the known one. Moreover, the proof is
different and, in fact, even more transparent.

\smallskip Briefly, the paper is organized as follows. In Section \ref{sec1}
we recall basic material on Wigner and Born-Jordan distributions, the
corresponding transforms, and we review the recent advances in the theory of
Born--Jordan quantization; for proofs and details we refer to Cordero et
al.\ \cite{cogoni15}, de Gosson \cite{TRANSAM,Springer}. In Section 3 we
study the Wigner and Born-Jordan distribution of superpositions of Gaussians
and we illustrate graphically the appearance of the interferences, and how
they are affected by rotations in the time-frequency plane. In Section 4 we
study in full generality the covariance property of the time-frequency
distributions in the Cohen class with respect to symplectic transformations.
In particular we prove the above mentioned characterization of the Wigner
transform and the associated Weyl operator.

\begin{notation}
We will use multi-index notation: $\alpha=(\alpha_{1},...,\alpha_{n})\in%
\mathbb{N}^{n}$, $|\alpha|=\alpha_{1}+\cdot\cdot\cdot+\alpha_{n}$, and $%
\partial_{x}^{\alpha}=\partial_{x_{1}}^{\alpha_{1}}\cdot\cdot\cdot
\partial_{x_{n}}^{\alpha_{n}}$. We denote by $\sigma$ the standard
symplectic form on the phase space $\mathbb{R}^{2n}\equiv\mathbb{R}%
^{n}\times \mathbb{R}^{n}$; the phase space variable is denoted $z=(x,p)$.
By definition $\sigma(z,z^{\prime})=Jz\cdot z^{\prime}$ where $J=%
\begin{pmatrix}
0 & I \\ 
-I & 0%
\end{pmatrix}
$.

We will use the notation $\widehat{x}_{j}$ for the operator of
multiplication by $x_{j}$ and $\widehat{p}_{j}=-i\hbar\partial/\partial
x_{j} $. These operators satisfy Born's canonical commutation relations $[%
\widehat{x}_{j},\widehat{p}_{j}]=i\hbar$.

The symplectic Fourier transform of a function $a(z)$ in phase space ${%
\mathbb{R}^{2n}}$ is normalized as 
\begin{equation}
\mathcal{F}_{\sigma }(a)(z)=a_{\sigma }(z)=\left( \tfrac{1}{2\pi \hbar }%
\right) ^{n}\int e^{-\frac{i}{\hbar }\sigma (z,z^{\prime })}a(z^{\prime
})dz^{\prime },\quad z\in \mathbb{R}^{2n}.  \label{sympft}
\end{equation}%
We will use the following variant of the Fourier transform:
\begin{equation}\label{FT}
\cF f(x)=\left( \tfrac{1}{2\pi \hbar }\right)^{n/2} \int e^{-%
\tfrac{i}{\hbar} p\cdot y} f(y) dy
\end{equation}
The translation operator in the time-frequency plane is defined by 
\begin{equation*}
\lbrack T(z_{0})a](z)=a(z-z_{0}),\quad z,z_{0}\in {\mathbb{R}^{2n}}.
\end{equation*}%
We denote by $\mathrm{Mp}(n,\mathbb{R})$ the metaplectic group, that is the
double covering of the symplectic group $\mathrm{Sp}(n,\mathbb{R})$. As is
well known, the elements of $\mathrm{Mp}(n,\mathbb{R})$ can be regarded as
unitary operators in $L^{2}(\mathbb{R}^{n})$ ($\mathrm{Mp}(n,\mathbb{R})$
has a faithful strongly continuous unitary representation in $L^2(\rd)$). We will %
{reserve} the notation $\widehat{S}\in \mathrm{Mp}(n,\mathbb{R%
})$ for a metaplectic operator and $S=\pi ^{\mathrm{Mp}}(\widehat{S})\in 
\mathrm{Sp}(n,\mathbb{R})$ for its projection in $\mathrm{Sp}(n,\mathbb{R})$%
; see \cite[Chapter 7]{Birkbis} for more details.
\end{notation}

\section{Wigner and Born--Jordan distributions and associated
pseudodifferential operators\label{sec1}}

\subsection{Wigner and Born-Jordan distributions}

We recall that the Wigner-Moyal transform of $f,g\in\mathcal{S}(\mathbb{R}%
^{n})$ is defined by%
\begin{equation}
W(f,g)(z)=\left( \tfrac{1}{2\pi\hbar}\right) ^{n}\int_{\mathbb{R}^{n}}e^{-%
\tfrac{i}{\hbar}p\cdot y}f(x+\tfrac{1}{2}y)g^{\ast}(x-\tfrac{1}{2}y)dy,\quad
z=(x,p).  \label{wimo}
\end{equation}
We also set 
\begin{equation*}
Wf=W(f,f).
\end{equation*}
Let $\widehat{T}(z_{0})=e^{-\frac{i}{\hslash}\sigma(\hat{z},z_{0})}$ be the
Heisenberg operator already defined in \eqref{heiwe}. We have the
translation formula%
\begin{equation*}
W(\widehat{T}(z_{0})f)(z)=T(z_0)Wf(z)=Wf(z-z_{0}).
\end{equation*}

Recall that the Wigner transform satisfies the following translation
property:  for every $f\in L^{2}(\mathbb{R}^{n})$ and $z_{0}\in \mathbb{R}%
^{2n}$ we have%
\begin{equation}
W(\widehat{T}(z_{0})f,\widehat{T}(z_{0})g)(z)=T(z_{0})W(f,g)(z).  \label{wt8}
\end{equation}%
In particular 
\begin{equation}
W(\widehat{T}(z_{0})f)=T(z_{0})Wf.  \label{trans1}
\end{equation}%
More generally, the following result holds true (\cite{Birk,folland}):

\begin{proposition}
\label{transam}If $f,g\in L^{2}(\mathbb{R}^{n})$, then 
\begin{equation}
W(\widehat{T}(z_{0})f,\widehat{T}(z_{1})g)(z)=e^{-\frac{i}{\hbar }[\sigma
(z,z_{0}-z_{1})+\frac{1}{2}\sigma (z_{0},z_{1})]}W(f,g)(z-\langle z\rangle )
\label{wt}
\end{equation}%
where $\langle z\rangle =\frac{1}{2}(z_{0}+z_{1})$.
\end{proposition}

The Wigner distribution enjoys covariance property with respect to all
(linear) symplectic transformation of the time-frequency plane (\cite%
{Birkbis} p. 151):

\begin{theorem}\label{teo22}
\label{invarianza} Let $\widehat{S}\in\mathrm{Mp}(n,\mathbb{R})$ and $S=\pi^{%
\mathrm{Mp}}(\widehat{S})\in\mathrm{Sp}(n,\mathbb{R})$. We have 
\begin{equation*}
W(\widehat{S}f)(z)=Wf(S^{-1}z), \quad f\in\mathcal{S}(\mathbb{R}^n), \quad z
\in{\mathbb{R}^{2n}}.
\end{equation*}
\end{theorem}

The Born-Jordan distribution, first introduced in \cite{Cohen1}, is defined
by 
\begin{equation}  \label{bj}
Q_{\mathrm{BJ}} f= Wf \ast (\theta_{\mathrm{BJ}})_{\sigma}
\end{equation}
where $\theta_{\mathrm{BJ}}$ is Cohen's kernel function defined by 
\begin{equation}  \label{e4}
\theta_{\mathrm{BJ}}(x,p)=\func{sinc}\left( \frac{px}{2\hbar}\right).
\end{equation}
(Recall that the function $\func{sinc}$ is defined by $\func{sinc}u=\sin u/u$
for $u\neq0$ and $\func{sinc}0=1$.)

\subsection{Weyl and Born-Jordan pseudodifferential operators}

The Weyl operator $\widehat{A}_{\mathrm{W}}=\limfunc{Op}_{\mathrm{W}}(a)$
with symbol $a\in\mathcal{S}^{\prime }({\mathbb{R}^{2n}})$ is given by the
familiar formula%
\begin{equation}
\widehat{A}_{\mathrm{W}}=\left( \tfrac{1}{2\pi\hbar}\right) ^{n}\int
a_{\sigma}(z)\widehat{T}(z)dz,  \label{AW1}
\end{equation}
where $a_{\sigma}$ is the symplectic Fourier transform of $a$ defined in %
\eqref{sympft}.

The Weyl transform can also be written in terms of the Wigner distribution
by the formula 
\begin{equation*}
\langle \widehat{A}_{\mathrm{W}} f,g\rangle=\langle a, W(g,f)\rangle.
\end{equation*}
The Born-Jordan operator $\widehat{A}_{\mathrm{BJ}}$ with (Born-Jordan)
symbol $a\in\mathcal{S}^{\prime }({\mathbb{R}^{2n}})$ is given by 
\begin{equation*}
\widehat{A}_{\mathrm{BJ}}=\mathrm{Op}_{\mathrm{BJ}}(a)=\limfunc{Op}%
\nolimits_{\mathrm{W}}(a\ast(\theta_{\mathrm{BJ}})_{\sigma}),
\end{equation*}
and represents an extension to arbitrary symbols of the first quantization
rule introduced in the literature \cite{bj} in the case of polynomial
symbols. Indeed {for the case} of monomial symbols one proves
(de Gosson \cite{TRANSAM}, de Gosson and Luef \cite{golu1}) the following
result.

\begin{proposition}
Let $a\in\mathcal{S}^{\prime}(\mathbb{R}^{2n})$ be a symbol. The restriction
of the linear operator $\widehat{A}_{\mathrm{BJ}}$ to monomials $p^{s}x^{r}$
(for $n=1)$ is given by the Born--Jordan rule 
\begin{equation}
p^{s}x^{r}\overset{\mathrm{BJ}}{\longrightarrow}\frac{1}{s+1}%
\sum_{\ell=0}^{s}\widehat{p}^{s-\ell}\widehat{x}^{r}\widehat{p}^{\ell} .
\label{bj1}
\end{equation}
\end{proposition}

In general it follows from formula (\ref{AW1}) that $\widehat{A}_{\mathrm{BJ}%
}$ is alternatively given by 
\begin{equation*}
\widehat{A}_{\mathrm{BJ}}=\left( \tfrac{1}{2\pi\hbar}\right) ^{n}\int
a_{\sigma}(z)\theta_{\mathrm{BJ}} (z)\widehat{T}(z)dz.
\end{equation*}

In \cite{cogoni15} we have proven that every Weyl operator has a
Born--Jordan symbol; equivalently, every linear continuous operator $%
\widehat{A}:\mathcal{S}(\mathbb{R}^{n})\longrightarrow\mathcal{S}^{\prime}(%
\mathbb{R}^{n})$ is a Born--Jordan pseudodifferential operator $\widehat{A}_{%
\mathrm{BJ}}=\limfunc{Op}_{\mathrm{BJ}}(b)$ for some symbol $b\in\mathcal{S}%
^{\prime}(\mathbb{R}^{2n})$. The proof of this property is far from being
trivial, since it amounts to solving a division problem: in view of
Schwartz's kernel theorem (see \textit{e.g}. H\"{o}rmander \cite{69}) every
such operator $\widehat{A}$ can be written as $\widehat{A}=\limfunc{Op}%
\nolimits_{\mathrm{W}}(a)$ for some $a\in\mathcal{S}^{\prime}(\mathbb{R}%
^{2n})$: it suffices to take 
\begin{equation}
a(x,p)=\int e^{-\frac{i}{\hbar}py}K(x+\tfrac{1}{2}y,x-\tfrac{1}{2}y)d^{n}y
\label{AK9}
\end{equation}
where $K$ is the kernel of $\widehat{A}$. Now, if we also want to show that
there exists a symbol $b$ such that $\widehat{A}_{\mathrm{BJ}}=%
\limfunc{Op}_{\mathrm{BJ}}(b)$ then, since we have by definition $\widehat{A}%
_{\mathrm{BJ}}=\limfunc{Op}\nolimits_{\mathrm{W}}(b\ast(\theta_{\mathrm{BJ}%
})_\sigma)$, we have to solve the equation $b\ast(\theta_{\mathrm{BJ}%
})_\sigma=a$, that is, taking symplectic Fourier transforms,%
\begin{equation*}
b_{\sigma}(z)\func{sinc}\left( \frac{px}{2\hbar}\right) =a_{\sigma }(z).
\end{equation*}
We are thus confronted with a division problem, the difficulty coming from
the fact that we have $\func{sinc}\left( px/2\hbar\right) =0$ for all $(x,p)$
such that $px=2N\pi\hbar$ for a non-zero integer $N$. Nevertheless, one
proves (\cite{cogoni15}) that such a division is always possible in $%
\mathcal{S}^{\prime }({\mathbb{R}^{2n}})$.

\section{Squeezed states and interferences}

We collect in this section some material about the (cross-)Wigner transforms
of generalized Gaussian functions (the \textquotedblleft squeezed
states\textquotedblright\ familiar from quantum optics), and their
translates. For details see e.g. de Gosson \cite{Birkbis} and the references
therein. We then illustrate the phenomenon of the interferences for the sum
of four Gaussians in the diamond configuration, both for the Wigner and
Born-Jordan distribution.

\subsection{Generalized Gaussians and their Wigner transforms}

We will use the following well-know generalized Fresnel formula giving the
Fourier transform of Gaussians:

\begin{lemma}
Let $\phi_{M}(x)=e^{-\tfrac{1}{2\hbar}Mx^{2}}$ where $M=X+iY$ is a symmetric
complex $n\times n$ matrix such that $X=\func{Re}M>0$. We have%
\begin{equation}
\cF\phi_{M}(x)=(\det M)^{-1/2}\phi_{M^{-1}}(x)  \label{fofol8}
\end{equation}
with $\cF$ is defined in \eqref{FT},
where $(\det M)^{-1/2}$ is given by the formula%
\begin{equation*}
(\det M)^{-1/2}=\lambda_{1}^{-1/2}\cdot\cdot\cdot\lambda_{m}^{-1/2}
\end{equation*}
the numbers $\lambda_{1}^{-1/2},...,\lambda_{m}^{-1/2}$ being the square
roots with positive real parts of the eigenvalues $\lambda_{1}^{-1},...,%
\lambda _{m}^{-1}$ of $M^{-1}$.
\end{lemma}

From now on we denote by $\psi_{M}^{\hbar}$ the normalized Gaussian function
defined by 
\begin{equation}
\psi_{M}^{\hbar}(x)=\left( \tfrac{1}{\pi\hbar}\right) ^{n/4}(\det
X)^{1/4}e^{-\tfrac{1}{2\hbar}Mx^{2}}  \label{gausshud}
\end{equation}
where $M$ is as above. Gaussians of this type are often called
\textquotedblleft squeezed coherent states\textquotedblright; the reason is
that they can be obtained from the standard coherent state $\psi_{0}^{\hbar
}(x)=(\pi\hbar)^{-n/4}e^{-|x|^{2}/2\hbar}$ using a metaplectic
transformation.

The following result shows that $W\psi_{M}^{\hbar}$ is in fact a phase space
Gaussian of a very special type:

\begin{proposition}
\label{wiggauss}Let $M=X+iY$ and $\psi_{M}^{\hbar}$ be defined as above.

(i) The Wigner transform of the squeezed state $\psi_{M}^{\hbar}$ is the
phase space Gaussian 
\begin{equation}
W\psi_{M}^{\hbar}(z)=\left( \tfrac{1}{\pi\hbar}\right) ^{n}e^{-\tfrac {1}{%
\hbar}Gz^{2}}  \label{phagauss}
\end{equation}
where $G$ is the symmetric matrix%
\begin{equation}
G=%
\begin{pmatrix}
X+YX^{-1}Y & YX^{-1} \\ 
X^{-1}Y & X^{-1}%
\end{pmatrix}
;  \label{gsym}
\end{equation}
(ii) We have $G\in\limfunc{Sp}(n)$; in fact $G=S^{T}S$ where%
\begin{equation}
S=%
\begin{pmatrix}
X^{1/2} & 0 \\ 
X^{-1/2}Y & X^{-1/2}%
\end{pmatrix}
\label{bi}
\end{equation}
is a symplectic matrix.
\end{proposition}

\begin{proof}[\textit{Proof of (i)}]
Set $C(X)=\left( \pi\hbar\right) ^{n/4}(\det X)^{1/4}$. By definition of the
Wigner transform we have%
\begin{equation}
W\psi_{M}^{\hbar}(z)=\left( \tfrac{1}{2\pi\hbar}\right) ^{n}C(X)^{2}\int_{%
\mathbb{R}^{n}}e^{-\frac{i}{\hbar}p\cdot y}e^{-\frac{1}{2\hbar}F(x,y)}dy
\label{wigcoh11}
\end{equation}
where the phase $F$ is defined by%
\begin{align*}
F(x,y) & =(X+iY)(x+\tfrac{1}{2}y)^{2}+(X-iY)(x-\tfrac{1}{2}y)^{2} \\
& =2Xx\cdot x+2iYx\cdot y+\tfrac{1}{2}Xy\cdot y
\end{align*}
and hence%
\begin{equation*}
W\psi_{M}^{\hbar}(z)=\left( \tfrac{1}{2\pi\hbar}\right) ^{n}e^{-\frac {1}{%
\hbar}Xx^{2}}C(X)^{2}\int_{\mathbb{R}^{n}}e^{-\frac{i}{\hbar}(p+Yx)\cdot
y}e^{-\frac{1}{4\hbar}Xy^{2}}dy\text{.}
\end{equation*}
Using the Fourier transformation formula (\ref{fofol8}) above with $x$
replaced by $p+Yx$ and $M$ by $\frac{1}{2}X$ we get 
\begin{multline*}
\int_{\mathbb{R}^{n}}e^{-\frac{i}{\hbar}(p+Yx)\cdot y}e^{-\frac{1}{4\hbar }%
Xy\cdot y}dy=(2\pi\hbar)^{n/2}\left[ \det(\tfrac{1}{2}X)\right] ^{-1/2} \\
\times C(X)^{2}\exp\left[ -\tfrac{1}{\hbar}X^{-1}(p+Yx)\cdot(p+Yx)\right] .
\end{multline*}
On the other hand we have 
\begin{equation*}
(2\pi\hbar)^{n/2}\left[ \det(\tfrac{1}{2}X)\right] ^{-1/2}C(X)^{2}=\left( 
\tfrac{1}{\pi\hbar}\right) ^{n}
\end{equation*}
and hence%
\begin{equation*}
W\psi_{M}^{\hbar}(z)=\left( \tfrac{1}{\pi\hbar}\right) ^{n}e^{-\tfrac {1}{%
\hbar}Gz^{2}}
\end{equation*}
where 
\begin{equation*}
Gz^{2}=(X+YX^{-1})x\cdot x+2X^{-1}Yx\cdot p+X^{-1}p\cdot p\text{.}
\end{equation*}
\textit{Proof of (ii). }The symmetry of $G$ is of obvious, and so is the
factorization $G=S^{T}S$. One immediately verifies that $S^{T}JS=J$ hence $%
S\in\limfunc{Sp}(n,\mathbb{R})$ as claimed.
\end{proof}

In particular, when $\psi_{0}^{\hslash}$ is the standard coherent state one
recovers the standard formula%
\begin{equation}
W\psi_{0}^{\hslash}(z)=\left( \tfrac{1}{\pi\hbar}\right) ^{n}e^{-\frac {1}{%
\hbar}|z|^{2}}.  \label{wii8}
\end{equation}

\subsection{The cross-Wigner transform of a pair of Gaussians}

Let us now generalize formula (\ref{phagauss}) by calculating the
cross-Wigner transform $W(\psi_{M}^{\hbar},\psi_{M^{\prime}}^{\hbar})$ of a
pair of Gaussians of the type above.

Let $M$ be a complex matrix; We will denote by $\overline{M}$ its complex
conjugate: if $M=(m_{i,j})_{1\leq i,j\leq n}$ then $\overline{M}%
=(m_{i,j}^{\ast})_{1\leq i,j\leq n}$.

\begin{proposition}
Let $\psi_{M}^{\hbar}$ and $\psi_{M^{\prime}}^{\hbar}$ be Gaussian functions
of the type (\ref{gausshud}). We have%
\begin{equation}
W(\psi_{M}^{\hbar},\psi_{M^{\prime}}^{\hbar})(z)=\left( \tfrac{1}{\pi\hbar }%
\right) ^{n}C_{M,M^{\prime}}e^{-\frac{1}{\hbar}Fz^{2}}  \label{wiwi8}
\end{equation}
where $C_{M,M^{\prime}}$ is a constant given by 
\begin{equation}
C_{M,M^{\prime}}=(\det XX^{\prime})^{1/4}\det\left[ \tfrac{1}{2}(M+\overline{%
M^{\prime}})\right] ^{-1/2}  \label{constant8}
\end{equation}
and $F$ is the symmetric complex matrix given by%
\begin{equation}
F=%
\begin{pmatrix}
2\overline{M^{\prime}}(M+\overline{M^{\prime}})^{-1}M & -i(M-\overline {%
M^{\prime}})(M+\overline{M^{\prime}})^{-1} \\ 
-i(M+\overline{M^{\prime}})^{-1}(M-\overline{M^{\prime}}) & 2(M+\overline {%
M^{\prime}})^{-1}%
\end{pmatrix}
.  \label{wigf8}
\end{equation}
\end{proposition}

\begin{proof}
We have%
\begin{equation*}
W(\psi_{M}^{\hbar},\psi_{M^{\prime}}^{\hbar})(z)=C(X,X^{\prime})\int_{%
\mathbb{R}^{n}}e^{-\frac{i}{\hbar}py}e^{-\frac{1}{2\hbar}\Phi(x,y)}dy
\end{equation*}
where the functions $C$ and $\Phi$ are given by%
\begin{align*}
C(X,X^{\prime}) & =2^{-n}\left( \tfrac{1}{\pi\hbar}\right) ^{2n}(\det
XX^{\prime})^{1/4} \\
\Phi(x,y) & =M(x+\tfrac{1}{2}y)^{2}+\overline{M^{\prime}}(x-\tfrac{1}{2}%
y)^{2}\text{.}
\end{align*}
Let us evaluate the integral%
\begin{equation*}
I(z)=\int_{\mathbb{R}^{n}}e^{-\frac{i}{\hbar}py}e^{-\frac{1}{2\hbar}\Phi
(x,y)}dy, \quad z=(x,p)\in \mathbb{R}^{2n}.
\end{equation*}
We have%
\begin{equation*}
\Phi(x,y)=(M+\overline{M^{\prime}})x^{2}+\frac{1}{4}(M+\overline{M^{\prime}}%
)y^{2}+(M-\overline{M^{\prime}})x\cdot y
\end{equation*}
and hence%
\begin{equation*}
I(z)=e^{-\frac{1}{2\hbar}(M+\overline{M^{\prime}})x^{2}}\int_{\mathbb{R}%
^{n}}e^{-\frac{i}{\hbar}[p-\frac{i}{2}(M-\overline{M^{\prime}})x]\cdot y}e^{-%
\frac{1}{8\hbar}(M+\overline{M^{\prime}})y^{2}}dy.
\end{equation*}
Using the Fourier transformation formula (\ref{fofol8}) we get%
\begin{multline}
I(z)=(2\pi\hbar)^{n/2}\det\left[ \tfrac{1}{4}(M+\overline{M^{\prime}})\right]
^{-1/2}  \notag \\
\times\exp\left( -\frac{1}{2\hbar}\left[ (M+\overline{M^{\prime}})x^{2}+4(M+%
\overline{M^{\prime}})^{-1}\left( p-\tfrac{1}{2}(M-\overline {M^{\prime}}%
)x\right) ^{2}\right] \right) .
\end{multline}
A straightforward calculation shows that 
\begin{equation*}
\tfrac{1}{2}(M+\overline{M^{\prime}})x^{2}+4(M+\overline{M^{\prime}}%
)^{-1}\left( p-\tfrac{1}{2}(M-\overline{M^{\prime}})x\right) ^{2}=Fz\cdot z
\end{equation*}
where $F$ is the matrix%
\begin{equation}
\begin{pmatrix}
K & -i(M-\overline{M^{\prime}})(M+\overline{M^{\prime}})^{-1} \\ 
-i(M+\overline{M^{\prime}})^{-1}(M-\overline{M^{\prime}}) & 2(M+\overline {%
M^{\prime}})^{-1}%
\end{pmatrix}
\label{matrixm8}
\end{equation}
with left upper block 
\begin{equation*}
K=\tfrac{1}{2}\left[ M+\overline{M^{\prime}}-(M-\overline{M^{\prime}})(M+%
\overline{M^{\prime}})^{-1}(M-\overline{M^{\prime}})\right] .
\end{equation*}
Using the identity%
\begin{equation}
M+\overline{M^{\prime}}-(M-\overline{M^{\prime}})(M+\overline{M^{\prime}}%
)^{-1}(M-\overline{M^{\prime}})=4\overline{M^{\prime}}(M+\overline{M^{\prime
}})^{-1}M  \label{identity8}
\end{equation}
the matrix (\ref{matrixm8}) is given by (\ref{wigf8}). We thus have,
collecting the constants and simplifying the obtained expression 
\begin{equation*}
W(\psi_{M}^{\hbar},\psi_{M^{\prime}}^{\hbar})(z)=\left( \tfrac{1}{\pi\hbar }%
\right) ^{n}(\det XX^{\prime})^{1/4}\det\left[ \tfrac{1}{2}(M+\overline {%
M^{\prime}})\right] ^{-1/2}e^{-\frac{1}{\hbar}Fz^{2}}
\end{equation*}
which we set out to prove.
\end{proof}

\subsection{Superposition of squeezed coherent states}

We are now ready to use the above machinery to study the superposition of
squeezed coherent states, and the interferences which arise due the
quadratic nature of the Wigner distribution.

Let $f=\sum_{1\leq j\leq m}\lambda_{j}f_{j}$ be a finite linear
superposition of quantum states $f_j\in L^2(\mathbb{R}^n)$; an easy
computation shows that the Wigner distribution $Wf$ is given by%
\begin{equation}  \label{e1}
Wf =\sum_{j=1}^{m}|\lambda_j|^2 Wf_{j}+2\underset{k>\ell}{\func{Re}\sum
_{k=1}^{m}\sum_{\ell=1}^{m}}\lambda_k\bar{\lambda_\ell} W(f_{k},f_{\ell}).
\end{equation}

\begin{corollary}
Let $f=\sum_{1\leq j\leq m}\widehat{T}(z_{j})f_{j}$, $z_j\in {\mathbb{R}^{2n}%
}$, be a finite linear superposition of Heisenberg shifts of quantum states.
Then 
\begin{equation}  \label{e2}
Wf(z)=\sum_{j=1}^{m} Wf_{j}(z-z_j)+2\underset{k>\ell}{\func{Re}\sum
_{k=1}^{m}\sum_{\ell=1}^{m}} e^{-\frac{i}{\hbar} [\sigma(z,z_{k}-z_{\ell})+%
\frac{1}{2}\sigma(z_{k},z_{\ell})]} W(f_{k},f_{\ell})(z-\langle
z_{k,\ell}\rangle),
\end{equation}
where $\langle z_{k,\ell}\rangle:= \frac 12 (z_k+z_\ell)$.
\end{corollary}

\begin{proof}
The formula \eqref{e2} is a straightforward consequence of \eqref{e1} and %
\eqref{wt}.
\end{proof}

If we consider the standard coherent state $\psi_{0}^{\hbar
}(x)=(\pi\hbar)^{-n/4}e^{-|x|^{2}/2\hbar}$, from formula \eqref{phagauss} we
infer $W\psi_{0}^{\hbar }(z)=\left( \tfrac{1}{\pi\hbar}\right) ^{n}e^{-%
\tfrac {1}{\hbar}z^{2}}$. The superposition of Heisenberg shifts of the
standard coherent state $\psi_{0}^{\hbar }$ is then given by 
\begin{align}  \label{e3}
W&(\sum_{j=1}^{m}\widehat{T}(z_{j})\psi_{0}^{\hbar })(z) \\
&= \left( \tfrac{1}{\pi\hbar}\right) ^{n}\left(\sum_{j=1}^{m} e^{-\tfrac {1}{%
\hbar}|z-z_j|^{2}}+2\underset{k>\ell}{\sum _{k=1}^{m}\sum_{\ell=1}^{m}}
\cos\{\frac{1}{\hbar} [\sigma(z,z_{k}-z_{\ell})+\frac{1}{2}%
\sigma(z_{k},z_{\ell})]\} e^{-\tfrac {1}{\hbar}|z-\langle
z_{k,\ell}\rangle|^{2}}\right).  \notag
\end{align}

%\begin{figure}[htbp]
%\begin{center}
%\includegraphics[width=12cm, height=12cm]{Specdb}
%\caption{Spectrogram, logarithmic  amplitude}
%\label{SpFig}
%\end{center}
%\end{figure}
Figure \ref{WigFig} shows the Wigner transform of the superposition of four
quantum states (Gaussians), in rotated positions for $9$ steps between the
original position and the final position corresponding to a rotation by $%
45^{\circ}$. Both the nature of the interference terms as stated in the
previous corollary and the covariance property of the Wigner transform are
visible. 
\begin{figure}[htbp]
\begin{center}
\includegraphics[width=12cm, height=12cm]{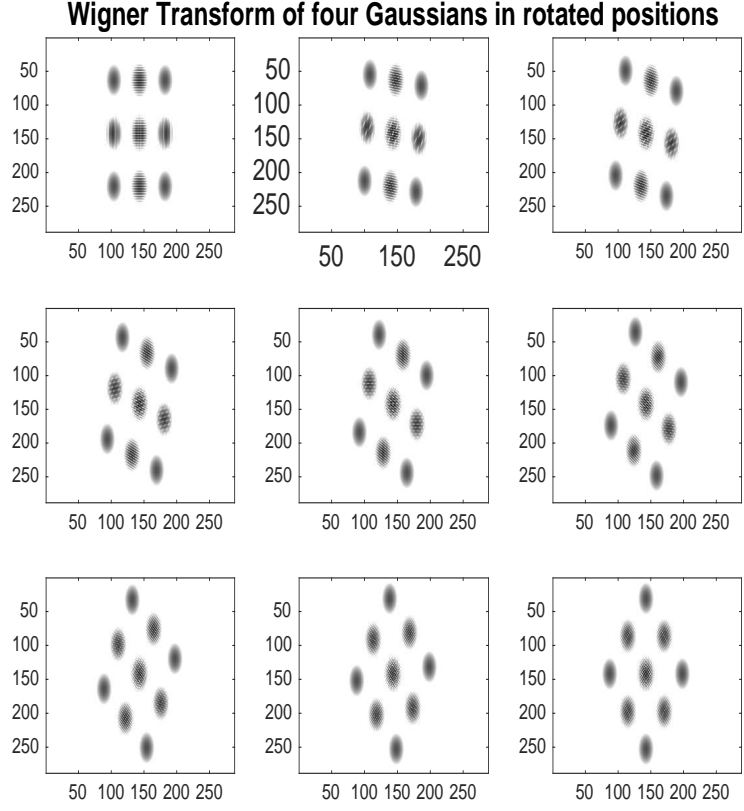}
\end{center}
\caption{Wigner transform, logarithmic amplitude}
\label{WigFig}
\end{figure}

\section{Symplectic covariance for the Cohen class}

A quadratic time-frequency representation $Q$ belongs to the Cohen's class
if it can be written as 
\begin{equation}  \label{Cohen}
Q f= Wf\ast \theta, \quad \forall f\in \mathcal{S}(\mathbb{R}^n)
\end{equation}
for a suitable kernel $\theta\in\mathcal{S}^{\prime }({\mathbb{R}^{2n}})$.

Observe that for a distribution in the Cohen's class, using \eqref{wt8} we
have, for every $f\in\mathcal{S}(\mathbb{R}^n)$, 
\begin{align*}
Q(\widehat{T}(z_{0})f)(z)&=(W(\widehat{T}(z_{0})f)\ast\theta)(z)
=((T(z_{0})Wf)\ast\theta)(z) \\
&=T(z_{0})(Wf\ast\theta)(z)=T(z_{0})Qf(z).
\end{align*}
Hence, for every $Q$ in the Cohen's class, we have the translation formula: 
\begin{equation}
Q(\widehat{T}(z_{0})f)(z)=Q f(z-z_0),
\end{equation}
as for the Wigner distribution.

Let us now study the behaviour of the Cohen's class under the action of
metaplectic operators.

\begin{proposition}[Symplectic covariance of the Cohen's class]
\label{pro36} Consider $\theta\in \mathcal{S}^{\prime }(\mathbb{R}^n)$ and $%
Q $ an element of the Cohen's class having kernel $\theta$, as in %
\eqref{Cohen}. For $\widehat{S}\in Mp(n,\mathbb{R})$ and $S = \pi^{Mp}(%
\widehat{S})$, we have 
\begin{equation}  \label{formula33}
Q(\widehat{S}f)(z)=Wf\ast(\theta\circ S)(S^{-1}z),\quad \forall f\in\mathcal{%
S}(\mathbb{R}^n),\quad z\in{\mathbb{R}^{2n}}.
\end{equation}
\end{proposition}

\begin{proof}
Recalling the symplectic covariance for the Wigner distribution (Theorem \ref{teo22}) 
\begin{equation*}
W(\widehat{S}f)(z)=Wf(S^{-1}z), \quad f\in\mathcal{S}(\mathbb{R}^n), \quad z
\in{\mathbb{R}^{2n}},
\end{equation*}
we can write, for any $f\in\mathcal{S}(\mathbb{R}^n)$, 
\begin{align*}
Q(\widehat{S}f)(z)&=(W(\widehat{S}f)\ast\theta)(z) \\
&=\int_{{\mathbb{R}^{2n}}} W(\widehat{S}f)(z-w)\theta(w)\, dw \\
&=\int_{{\mathbb{R}^{2n}}} Wf(S^{-1}(z-w))\theta(w)\, dw \\
&=\int_{{\mathbb{R}^{2n}}} Wf(u)\theta(z-Su)\, du \\
&=\int_{{\mathbb{R}^{2n}}} Wf(u)\theta(S(S^{-1}z-u))\, du \\
&=Wf\ast (\theta\circ S)(S^{-1}z)
\end{align*}
where the integrals must be understood in the sense of distributions (we
recall that $Wf\in\mathcal{S}({\mathbb{R}^{2n}})$ for $f\in\mathcal{S}(%
\mathbb{R}^n)$, cf. \cite[Theorem 11.2.5]{book}). Moreover, in the change of
variables we used $\det S=1$.
\end{proof}

As a consequence, we recover the results for the covariance of the
Born--Jordan distribution in \cite{TRANSAM}.

\begin{corollary}\label{cor42}
If $\theta=\theta_{\mathrm{BJ}}$ as in \eqref{e4}, we have the covariance of
the corresponding distribution $Q$ for all the symplectic matrices of the
type $S=J$ or $S=M_L=%
\begin{pmatrix}
L^{-1} & 0 \\ 
0 & L^t%
\end{pmatrix}
$, with $L\in \mathrm{GL}(2n,\mathbb{R})$.
\end{corollary}

The covariance behavior of the Born--Jordan distribution is illustrated in
Figure \ref{BJFig}, for the same set of rotated superposition of quantum
states we considered in the previous section. We observe that the
interference terms depend substantially on the underlying coordinate system.
In particular, choosing an appropriate rotation, the interference terms can
be significantly damped.

\begin{figure}[htbp]
\begin{center}
\includegraphics[width=12cm, height=12cm]{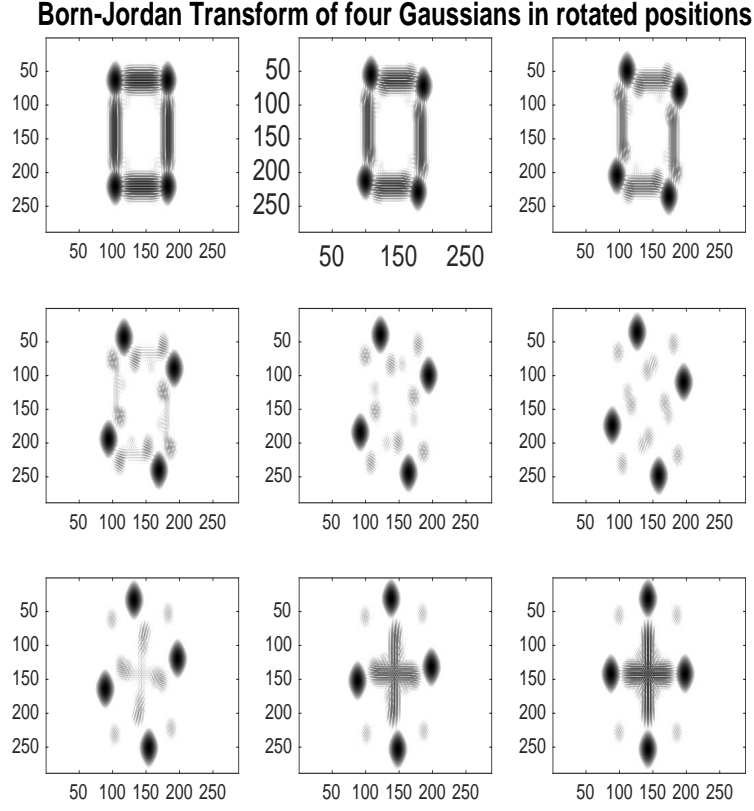}
\end{center}
\caption{Born-Jordan transform, logarithmic amplitude}
\label{BJFig}
\end{figure}

Formula \eqref{formula33} motivates the investigations of the temperate
distributions which are invariant with respect to the action of the linear
symplectic group. We need the following easy lemma.

\begin{lemma}
\label{lemmaff0} Let $\theta\in\mathcal{S}^{\prime }({\mathbb{R}^{2n}})$ and 
$A$ be a square $2n\times 2n$ real matrix. We have 
\begin{equation}  \label{ff0}
\frac{\mathrm{d}}{\mathrm{dt}} \theta(e^{tA}z)|_{t=0}=Az\cdot\nabla_z \theta.
\end{equation}

\begin{proof}
We have, for $\varphi\in\mathcal{S}({\mathbb{R}^{2n}})$, 
\begin{equation}  \label{ff01}
\langle \frac{1}{t}\big(\theta(e^{tA}\cdot)-\theta\big),\varphi\rangle=%
\langle \theta, \frac{1}{t}\big(e^{-t\,\mathrm{Tr}\,A}\varphi(e^{-tA}\cdot)-%
\varphi\big)\rangle.
\end{equation}
Now, as $t\to 0$ we have 
\begin{equation*}
\frac{1}{t}\big(e^{-t\,\mathrm{Tr}\,A}\varphi(e^{-tA}\cdot)-\varphi\big)\to -%
\mathrm{Tr}\,A\,\varphi-Az\cdot\nabla_z \varphi
\end{equation*}
in $\mathcal{S}({\mathbb{R}^{2n}})$, so that the left-hand side in %
\eqref{ff01} tends to 
\begin{align*}
\langle \theta, -\mathrm{Tr}\,A\,\varphi-Az\cdot\nabla_z
\varphi\rangle&=\langle -\mathrm{Tr}\,A\,\theta +\mathrm{div}\big(Az\,\theta%
\big),\varphi\rangle \\
&=\langle Az\cdot\nabla_z\theta,\varphi\rangle.
\end{align*}
\end{proof}
\end{lemma}

\begin{proposition}
\label{teosimpinv} Let $\theta\in\mathcal{S}^{\prime }({\mathbb{R}^{2n}})$
be such that 
\begin{equation*}
\theta\circ S=\theta
\end{equation*}
for every $S\in \mathrm{Sp}(n,\mathbb{R})$. Then 
\begin{equation*}
\theta=c_0+c_1\delta
\end{equation*}
for some $c_0,c_1\in\mathbb{C}$.
\end{proposition}

\begin{proof}
It suffices to consider the case where $S=e^{A}$ with $A\in \mathfrak{sp}(n,%
\mathbb{R})$ (the symplectic algebra, \textit{i.e.} the Lie algebra of $%
\mathrm{Sp}(n,\mathbb{R})$). Consider the one-parameter group of symplectic
matrices $e^{tA}\in \mathrm{Sp}(n,\mathbb{R})$, $t\in \mathbb{R}$. The
assumption $\theta \circ S=\theta $ implies that 
\begin{equation*}
\theta (e^{tA}z)=\theta (z)
\end{equation*}%
for every $A\in \mathfrak{sp}(n,\mathbb{R})$ and therefore 
\begin{equation*}
\frac{\mathrm{d}}{\mathrm{dt}}\theta (e^{tA}z)=0.
\end{equation*}%
On the other hand, by Lemma \ref{lemmaff0}, we have 
\begin{equation}
\frac{\mathrm{d}}{\mathrm{dt}}\theta (e^{tA}z)|_{t=0}=Az\cdot \nabla
_{z}\theta   \label{ff1}
\end{equation}%
and therefore 
\begin{equation}
Az\cdot \nabla _{z}\theta =0\quad \forall A\in \mathfrak{sp}(n,\mathbb{R}).
\label{ff2}
\end{equation}%
Now, for $z\not=0$ fixed, we have 
\begin{equation}
\mathcal{T}_{z}=\{Az:\ A\in \mathfrak{sp}(n,\mathbb{R})\}={\mathbb{R}^{2n}}%
,\quad z\not=0.  \label{ff3}
\end{equation}%
In fact $\mathcal{T}_{z}\simeq \{Az\cdot \nabla _{z}:\ A\in \mathfrak{sp}(n,%
\mathbb{R})\}$ is the tangent space at $z$ to the orbit of the action of $%
\mathrm{Sp}(n,\mathbb{R})$ on ${\mathbb{R}^{2n}}\setminus \{0\}$. Since the
action is transitive, the orbit is the whole ${\mathbb{R}^{2n}}\setminus
\{0\}$, and \eqref{ff3} follows.

As a consequence of \eqref{ff2} and \eqref{ff3} we have $\nabla_z \theta=0$
in ${\mathbb{R}^{2n}}\setminus\{0\}$, and therefore $\theta=c_0$ in ${%
\mathbb{R}^{2n}}\setminus\{0\}$. The distribution $\theta-c_0$ in ${\mathbb{R%
}^{2n}}$ is supported at $0$, so that 
\begin{equation}  \label{ff4}
\theta=c_0+c_1\delta+\sum_{1\leq|\alpha|\leq m} c^{\prime
}_{\alpha}\partial^{\alpha}_{z} \delta
\end{equation}
for some $m\geq 1$, $c^{\prime}_{\alpha}\in\mathbb{C}$. We have to show that
in fact the summation in \eqref{ff4} is zero. To this end we observe that,
taking the Fourier transform of both sides of the invariance property $%
\theta\circ S=\theta$, $S\in \mathrm{Sp}(n,\mathbb{R})$, we see that $%
\widehat{\theta}$ enjoys the same invariance, so that by the above argument
we have 
\begin{equation*}
\widehat{\theta}=c_2+v,
\end{equation*}
for some $c_2\in\mathbb{C}$, $v\in \mathcal{S}^{\prime }({\mathbb{R}^{2n}})$
supported at $0$. This is compatible with \eqref{ff4} only if the summation
in \eqref{ff4} is zero.
\end{proof}

\begin{remark}
An inspection of the above proof shows that we only need invariance with
respect to the symplectic matrices of the form $e^{tA}$, with $a\in A\in%
\mathfrak{sp}(n,\mathbb{R})$ and small $t$. However, this condition turns
out to be equivalent to the invariance with respect to the full symplectic
group, because any connected Lie group is generated by a neighborhood of the
identity.
\end{remark}

We now prove the characterization of the Wigner transform announced in
Introduction. For a quadratic map $Q$ we denote by $Q(f,g)$ its
corresponding sesquilinear map.

\begin{theorem}
\label{teonew} Consider a quadratic continuous time-frequency distribution $%
Q:L^2(\mathbb{R}^n)\to C_0({\mathbb{R}^{2n}})$, i.e. $Qf=Q(f,f)$ for a
sequilinear continuous map $Q:L^2(\mathbb{R}^n)\times L^2(\mathbb{R}^n)\to
C_0({\mathbb{R}^{2n}})$. Suppose that $Q$ enjoys:

\medskip (i) the covariance property with respect to translations in the
time-frequency plane: 
\begin{equation}  \label{uno-ibis}
Q(\widehat{T}(z_{0})f)(z)=Q f(z-z_0),\quad\forall f\in\mathcal{S}(\mathbb{R}%
^n),\ z_0\in\mathbb{R}^{2n};
\end{equation}

\medskip

(ii) the covariance property with respect to symplectic linear
transformations: for every $\widehat{S}\in\mathrm{Mp}(n,\mathbb{R})$, with $%
S=\pi^{\mathrm{Mp}}(\widehat{S})\in\mathrm{Sp}(n,\mathbb{R})$ 
\begin{equation}  \label{due-ibis}
Q(\widehat{S}f)(z)=Qf(S^{-1}z),\quad \forall f\in\mathcal{S}(\mathbb{R}^n),\
z\in\mathbb{R}^{2n}.
\end{equation}
Then 
\begin{equation*}
Qf=c\, Wf
\end{equation*}
for some constant $c\in\mathbb{C}$.
\end{theorem}

\begin{proof}
It follows from \cite[Theorem 4.5.1]{book} that the continuity assumption 
\begin{equation*}
|Q(f,g)(z)|\leq C\|f\|_{L^2(\mathbb{R}^n)}\|g\|_{L^2(\mathbb{R}^n)}, \quad
\forall f,g\in\mathcal{S}(\mathbb{R}^n),\ z\in{\mathbb{R}^{2n}}
\end{equation*}
together with \eqref{uno-ibis} imply that 
\begin{equation*}
Qf=Wf\ast\theta,\quad \forall f\in\mathcal{S}(\mathbb{R}^n)
\end{equation*}
for some distribution $\theta\in\mathcal{S}^{\prime }({\mathbb{R}^{2n}})$,
i.e.\ $Q$ is a distribution in the Cohen class.

Using this expression for $Q$, together with \eqref{formula33} and %
\eqref{due-ibis} we get 
\begin{equation*}
Wf\ast(\theta\circ S)(S^{-1}z)=Wf\ast\theta(S^{-1}z), \quad \forall f\in%
\mathcal{S}(\mathbb{R}^n),\ \forall z\in{\mathbb{R}^{2n}}.
\end{equation*}
Replacing $S^{-1}z$ by $z$ and taking the Fourier transform we get 
\begin{equation*}
\widehat{Wf}\, \widehat{\theta\circ S}=\widehat{Wf}\,\widehat{\theta},\quad
\forall f\in\mathcal{S}(\mathbb{R}^n).
\end{equation*}
If $f$ is a Gaussian function in $\mathbb{R}^n$, $\widehat{Wf}$ will be a
Gaussian itself in ${\mathbb{R}^{2n}}$, and in particular never vanishes.
Hence we obtain $\theta\circ S=\theta$ for every $S\in \mathrm{Sp}(n,\mathbb{%
R})$. From Proposition \ref{teosimpinv} we have $\theta=c_0+c_1\delta$ for
some $c_0,c_1\in\mathbb{C}$. Finally, since the distribution $%
Qf=Wf\ast\theta $ is assumed to tend to $0$ at infinity, for every $f\in%
\mathcal{S}(\mathbb{R}^n)$, it must be $c_0=0$ and we obtain the desired
result.
\end{proof}

We now provide a similar characterization for the Weyl transform. We begin
by characterizing the transform which enjoy a covariance property with
respect to time-frequency shifts, alias Heisenberg operators $\widehat{T}%
(z_{0})$.

\begin{theorem}
\label{teon} Consider a linear continuous mapping 
\begin{equation*}
\mathcal{S}^{\prime }({\mathbb{R}^{2n}})\to\mathcal{L}(\mathcal{S}(\mathbb{R}%
^n), \mathcal{S}^{\prime }(\mathbb{R}^n)),
\end{equation*}
say $a\longmapsto A_a$, satisfying the covariance property with respect to
the Heisenberg operators: 
\begin{equation}  \label{ff5}
\widehat{T}(-z_{0}) A_a \widehat{T}(z_{0})=A_{{T}(-z_{0})a} \quad \forall
z_0\in{\mathbb{R}^{2n}}.
\end{equation}
Then there exists a distribution $\theta\in\mathcal{S}^{\prime }({\mathbb{R}%
^{2n}})$, with $\widehat{\theta}$ smooth in ${\mathbb{R}^{2n}}$, such that 
\begin{equation}  \label{ff6}
\langle A_af,g\rangle=\langle a, W(g,f)\ast\theta\rangle \quad \forall
f,g\in \mathcal{S}(\mathbb{R}^d).
\end{equation}
\end{theorem}

\begin{proof}
By polarization it is sufficient to prove \eqref{ff6} when $f=g$. Now,
define the quadratic distribution 
\begin{equation*}
Q f(z)= \langle A_a f,f\rangle,\quad \mathrm{with}\
a(\cdot)=\delta(\cdot-z),\ z\in{\mathbb{R}^{2n}}.
\end{equation*}
By \eqref{ff5} it satisfies the covariance property 
\begin{equation*}
Q(\widehat{T}(z)f)=T(z) Q f
\end{equation*}
so that it follows from \cite[Theorem 4.5.1]{book} (or better from its
proof) that there exists a distribution $\theta\in\mathcal{S}^{\prime }({%
\mathbb{R}^{2n}})$ such that 
\begin{equation*}
Qf=Wf\ast \theta.
\end{equation*}
This proves \eqref{ff6} when $a(\cdot)=\delta(\cdot-z)$, $z\in{\mathbb{R}%
^{2n}}$.

Now, let $\mathcal{L}$ be the linear span of such symbols in $\mathcal{S}%
^{\prime }({\mathbb{R}^{2n}})$. Since the left-hand side of \eqref{ff6} is
continuous as a functional of $a\in\mathcal{S}^{\prime }({\mathbb{R}^{2n}})$%
, for fixed $f,g\in\mathcal{S}({\mathbb{R}^{2n}})$, we see that the
right-hand side extends to a continuous functional $\mathcal{S}^{\prime }({%
\mathbb{R}^{2n}})\to\mathbb{C}$, i.e.\ there exists a Schwartz function $b\in%
\mathcal{S}({\mathbb{R}^{2n}})$ such that 
\begin{equation*}
\langle a, W(g,f)\ast\theta\rangle =\langle a, b\rangle\quad \forall a\in%
\mathcal{L}
\end{equation*}
that is $W(g,f)\ast\theta=b $ is a Schwartz function. Hence the right-hand
side of \eqref{ff6} is also continuous $\mathcal{S}^{\prime }({\mathbb{R}%
^{2n}})\to\mathbb{C}$, as a functional of $a$, and therefore \eqref{ff6}
holds not only for $a\in \mathcal{L}$ but for every $a\in\mathcal{S}^{\prime
}({\mathbb{R}^{2n}})$, because $\mathcal{L}$ is dense in $\mathcal{S}%
^{\prime }({\mathbb{R}^{2n}})$ (cf.\ \cite[Lemma 7]{gossnic}).

Finally, we have already proved that $W(g,f)\ast\theta$, and therefore $%
\widehat{W(g,f)} \widehat{\theta}$, is a Schwartz function. For suitable
fixed Schwartz functions $f=g$ (e.g. a Gaussian) we have $\widehat{Wf}%
(w)\not=0$ for every $w\in{\mathbb{R}^{2n}}$, which implies that the
distribution $\widehat{\theta}$ is in fact smooth in ${\mathbb{R}^{2n}}$.
\end{proof}

As a consequence of these results we obtain the following new
characterization of the Weyl transform.

\medskip

\begin{theorem}
\label{teon2} Consider a linear continuous mapping 
\begin{equation*}
\mathcal{S}^{\prime }({\mathbb{R}^{2n}})\to\mathcal{L}(\mathcal{S}(\mathbb{R}%
^n), \mathcal{S}^{\prime }(\mathbb{R}^n)),
\end{equation*}
say $a\longmapsto A_a$, satisfying

\medskip (i) the covariance property \eqref{ff5} with respect to Heisenberg
operators;

\medskip

(ii) the covariance property with respect to metaplectic operators: for $%
\widehat{S}\in {\rm Mp}(n,\mathbb{R})$ and $S = \pi^{Mp}(\widehat{S})\in\mathrm{Sp}%
(n,\mathbb{R})$, 
\begin{equation}  \label{ff5bis}
\widehat{S}^{-1} A_a \widehat{S}=A_{a\circ S}.
\end{equation}
Then, for some $c\in\mathbb{C}$ we have 
\begin{equation}  \label{ff9}
\langle A_a f,g\rangle=c \langle a,W(g,f)\rangle\quad\forall f,g\in\mathcal{S%
}(\mathbb{R}^n).
\end{equation}
If in addition we have $A_a=I$ for $a=1$, then $c=1$ in \eqref{ff9}.
\end{theorem}

\begin{proof}
Again by polarization it is sufficient to prove \eqref{ff9} for $f=g$.

By Theorem \ref{teon} we have 
\begin{equation}  \label{ff10}
\langle A_af,f\rangle=\langle a, Wf\ast\theta\rangle \quad \forall f,g\in 
\mathcal{S}(\mathbb{R}^d)
\end{equation}
for some distribution $\theta\in\mathcal{S}^{\prime }({\mathbb{R}^{2n}})$
with $\widehat{\theta}$ smooth in ${\mathbb{R}^{2n}}$. Let $Qf:=Wf\ast\theta$%
. We have therefore, by Proposition \ref{pro36}, 
\begin{equation*}
\langle A_a \widehat{S}f,\widehat{S}f\rangle=\langle a,Q(\widehat{S}
f)\rangle=\langle a,[Wf\ast(\theta\circ S)]\circ S^{-1}\rangle.
\end{equation*}
On the other hand, using {\it (ii)} we have 
\begin{equation*}
\langle A_a \widehat{S}f,\widehat{S}g\rangle=\langle a\circ
S,Qf\rangle=\langle a,Qf\circ S^{-1}\rangle.
\end{equation*}
Hence it must be $Wf\ast\theta=Wf\ast(\theta\circ S)$ for every $f\in%
\mathcal{S}(\mathbb{R}^n)$. As in the proof of Theorem \ref{teonew} we
deduce that $\theta\circ S=\theta$ for every $S\in {\rm Sp}(n,\mathbb{R})$.
Proposition \ref{teosimpinv} implies that $\theta=c_0+c_1\delta$, for some $%
c_0,c_1\in\mathbb{C}$, but $\widehat{\theta}$ is smooth in ${\mathbb{R}^{2n}}
$, so that $c_0=0$. We have therefore $\theta=c\delta$ in \eqref{ff10},
which gives \eqref{ff9}.

The last part of the statement is clear, because $\langle f,g\rangle=\langle
1, W(g,f)\rangle$.
\end{proof}

\begin{remark}
\label{rem00} The result in Theorem \ref{teon2} can be compared with the
known characterization of the Weyl transform. Is is proved in \cite[Sections
7.5-7.6, pages 578-579]{stein} and \cite[Theorem 30.2]{wong} that the
conclusion of Theorem \ref{teon2} holds (with $c=1$) if one assumes

\medskip (i)' $A_a f= af$ for $a=a(x)\in L^\infty(\mathbb{R}^n)$

\medskip\noindent \textit{and} (ii).

We emphasize that, instead, in Theorem \ref{teon2}, (i) and (ii) together
amount to assuming the covariance with respect to the {\rm extended}
symplectic group.
\end{remark}

\section*{Acknowledgments}

E. Cordero and F. Nicola were partially supported by the Gruppo Nazionale
per l'Analisi Matematica, la Probabilit\`{a} e le loro Applicazioni (GNAMPA)
of the Istituto Nazionale di Alta Matematica (INdAM). M. de Gosson was
supported by the Austrian Research Agency FWF (Grant number P 27773). M. D\"orfler  has been supported by the Vienna Science and Technology Fund (WWTF)
through project MA14-018.

\end{document}